\documentclass[11pt]{article}
\usepackage{amsmath,amsthm,amsfonts,amssymb,bm,wasysym}
\usepackage{epsfig}
\usepackage[usenames]{color}
\usepackage{verbatim}
\usepackage{hyperref}
\usepackage{multicol}
\usepackage{comment}
\usepackage{float}
\usepackage{graphicx}
\usepackage[utf8]{inputenc}

\parskip \medskipamount
\newtheorem{theorem}{Theorem}[section]

\numberwithin{equation}{section}
\newtheorem{lemma}[theorem]{Lemma}
\newtheorem{proposition}[theorem]{Proposition}

\newtheorem{conjecture}[theorem]{Conjecture}
\numberwithin{equation}{section}

\def\N{\mathbb{N}}
\def\Z{\mathbb{Z}}
\def\Q{\mathbb{Q}}
\def\R{\mathbb{R}}

\def\E{\mathbb{E}}

\def\bP{\mathbb{P}}

\def\CC{\mathcal{C}}

\def\FF{\mathcal{F}}

\def\V{\mathcal{V}}

\def\RR{\mathcal{R}}
\def\OO{\mathcal{O}}

\renewcommand{\phi}{\varphi}
\renewcommand{\epsilon}{\varepsilon}
\def\reff#1{(\ref{#1})}

\title{\bf On the nature of the 
Swiss cheese in dimension 3.}

\begin{document}

\author{Amine Asselah \thanks{
Universit\'e Paris-Est Cr\'eteil; amine.asselah@u-pec.fr} \and
Bruno Schapira\thanks{Aix-Marseille Universit\'e, CNRS, Centrale Marseille, I2M, UMR 7373, 13453 Marseille, France;  bruno.schapira@univ-amu.fr} }
\date{}

\maketitle

\begin{abstract}
We study scenarii linked with the {\it Swiss cheese picture} in dimension three 
obtained when two random walks are forced to meet often,
or when one random walk is forced to squeeze its range.
In the case of two random walks, we show that they most likely
meet in a region of {\it optimal density}. In the case
of one random walk, we show that
a small range is reached by a strategy uniform in time.
Both results rely on an original inequality
estimating the cost of visiting sparse sites, and in the case of one random walk on the precise Large Deviation Principle of van den Berg, Bolthausen and den Hollander \cite{BBH}, including their sharp estimates
 of the rate functions in the neighborhood of the origin. \\

MSC 2010 \emph{subject classifications.} Primary 60F10, 60G50.
\end{abstract}

\section{Introduction.}
In this note, we are concerned with 
describing the geometry of the range of a random walk 
on $\Z^3$, when forced to having a small 
volume, deviating from its mean by a small fraction of it, 
or to intersecting often the range of another independent random walk.

These issues were raised in two landmark
papers of van den Berg, Bolthausen, and den Hollander (referred to as BBH in
the sequel) written two decades ago \cite{BBH} \cite{BBH2}.
Both papers dealt with the continuous counterpart of the range of a random walk,
the Wiener sausage. They showed a Large Deviation Principle,
in two related contexts: (i) in \cite{BBH} for the downward
deviation of the volume of the sausage, 
(ii) in \cite{BBH2} for the upward deviation of the 
volume of the intersection of two independent sausages. They also 
expressed the rate function with a variational formula.

Their sharp asymptotics are followed 
with a heuristic description of the optimal scenario 
dubbed the {\it Swiss cheese} picture where, in case (i), the
{\it Wiener sausage covers only part of the space leaving
random holes whose sizes are of order 1 and whose density varies
on space scale $t^{1/d}$}, and in case (ii) both Wiener 
sausages form apparently independent Swiss cheeses. 
However, they acknowledge that to show
that conditioned on the deviation, the sausages {\it actually follow
the Swiss cheese strategy requires substantial extra work}.

Remarkably the Swiss cheese heuristic also highlight a 
crucial difference between dimension $3$ and dimensions $5$ and higher. 
Indeed, in dimension three the typical scenario is time homogeneous, 
in the sense that the Wiener sausage considered up to time $t$, 
would spend all its time 
localized in a region of typical scale $(t/\varepsilon)^{1/3}$, 
filling a fraction of order $\varepsilon$ of every volume element, 
when the deviation of the volume occurs by a 
fraction $\varepsilon$ of the mean. On the other hand 
in dimension $5$ and higher the typical scenario would be
time-inhomogeneous: the Wiener sausage would localize 
in a smaller region of space with scale of order $(\varepsilon t)^{1/d}$, 
only during a fraction $\varepsilon$ of its time, and
therefore would produce a localization region where
the density is of order one, no matter how small $\epsilon$ is. 

Recently Sznitman \cite{S} suggested 
that the Swiss cheese could be described 
in terms of so-called tilted Random Interlacements. 
In the same time, in \cite{AS17} we obtained a first result in the discrete setting, 
which expressed the folding of the range of a random walk in terms of its capacity. 
More precisely we showed that a positive fraction of the path, considered up to some time $n$, was 
spent in a subset having 
a volume of order $n$ and a capacity of order $n^{1- 2/d}$, 
that is to say of the same order as the capacity of a ball with volume $n$. 

In this note, we present a simple and powerful estimate on the
probability a random walk visits sites which are 
{\it far from each others}, see Proposition \ref{lemLambda} below. 
We then deduce two applications in 
dimension three that we state vaguely as follows.
\begin{itemize}
\item One random walk, when forced to have a small range, 
{\it folds in a time-homogeneous way}.
\item Two random walks, when forced to meet often, 
do so in a region of {\it optimal density}.
\end{itemize}
To state our results more precisely, we introduce some notation. 
We denote by $\{S_n\}_{n\ge 0}$ the simple random walk on $\Z^d$ (in most of the paper $d=3$),  and by 
$\RR(I):=\{S_n\}_{n\in I}$, its range in 
a time window $I\subset \N$, which we simply write as $\RR_n$, when $I$ is the interval $I_n:=\{0,\dots,n\}$.
We let $Q(x,r):=(x+[-r,r[^d)\cap \Z^d$, be the cube of side length $2r$, and then define 
regions of {\it high density} as follows. First we define the (random) set of centers, depending on a density $\rho$, a scale $r$, and a time window $I$ as: 
\begin{equation*}
\CC(\rho,r,I):= 
\{x\in 2r\Z^d\, :\, |\RR(I)\cap Q(x,r)|\ge \rho \cdot |Q(x,r)|\},
\end{equation*}
and then the corresponding region 
\begin{equation*}
\V(\rho,r,I):= \bigcup_{x\in \CC(\rho,r,I)} Q(x,r),
\end{equation*}
which we simply write $\V_n(\rho,r)$, when $I=I_n$. 
Thus, $\RR(I)\cap \V(\rho,r,I)$ is the set of visited sites around
which the range has density, on a scale $r$, above $\rho$ in the time window $I$.

Our first result concerns the problem of forcing a single random walk having a small range. 
For $\varepsilon\in (0,1)$, we denote by $\Q_n^\varepsilon$ the law of the random walk conditionally on the 
event~$\{|\RR_n|-\E[|\RR_n|] \le - \varepsilon n\}$. 

\begin{theorem}\label{theo-appli-2} 
Assume $d=3$.
There exist positive constants $\beta$, $K_0$ and $\varepsilon_0$, 
such that for any $\varepsilon\in (0,\varepsilon_0)$, and 
any $1\le r\le \varepsilon^{5/6}  n^{1/3}$, 
\begin{equation*}
\lim_{n\to \infty} 
\Q_n^\varepsilon\left[
|\RR(I)\cap \RR(I^c)\cap \V_n(\beta \varepsilon,r)| 
\ge \frac{\varepsilon}{8} |I|,
\ \text{for all intervals } I\subseteq I_n,\text{ with } |I|=\lfloor K_0\varepsilon n\rfloor \right]=1.
\end{equation*}
\end{theorem}

This result expresses the fact that under $\Q_n^\varepsilon$, 
in any time interval of length 
of order $\varepsilon n$, the random walk intersects the 
other part of its range a fraction $\varepsilon$ of its  time, which is 
in agreement with the intuitive idea that, if during some time interval, the walk moves in a region with density of visited sites of order $\varepsilon$, it should intersect it 
a fraction $\varepsilon$ of its time. Note that it brings complementary information to the results obtained in \cite{AS17}, where it was shown that for some positive constants $\alpha$, $\beta$, $C$, and $\varepsilon_0\in (0,1)$, for any $\varepsilon\in (0,\varepsilon_0)$, and any
$C\varepsilon^{-5/9}n^{2/9}\log n \le r
\le \frac 1C\sqrt \varepsilon n^{1/3}$, 
\begin{equation}\label{AS17.3}
\lim_{n\to \infty} \, \Q_n^\varepsilon
\left[\exists \V \subseteq \V_n(\beta \varepsilon,r)\, :\, |\V\cap \RR_n|\ge \alpha  n,\text{ and }  \text{cap}(\V)
\le C(\frac n \varepsilon)^{1/3}\right]= 1,
\end{equation}
with $\text{cap}(\Lambda)$ being the capacity 
of a finite set $\Lambda\subset \Z^d$ 
(see for instance \cite{LL} for a definition). 
This is in a sense stronger than the result of Theorem \ref{theo-appli-2}, which only gives (by summation over disjoint intervals) that 
with high $\Q_n^\varepsilon$-probability, $|\RR_n\cap \V_n(\beta \varepsilon,r)| \ge \frac 18 \varepsilon n$.   
On the other hand \eqref{AS17.3} says nothing on the distribution of the times when the random walk visits the sets $\V_n(\beta \varepsilon,r)$,   
while Theorem \ref{theo-appli-2} 
shows that they are in a sense uniformly distributed. 
Both results are proved using different techniques: 
while \eqref{AS17.3} was obtained by using only elementary arguments, 
Theorem \ref{theo-appli-2} relies on the sharp and 
intricate results of \cite{BBH} (which have been obtained in the 
discrete setting by Phetpradap in his thesis \cite{P}).

Our second result concerns the problem of intersection of two independent ranges. For $n\ge 1$, and $\rho\in (0,1)$, we denote by $\widetilde \Q_n^\rho$ the law of two independent ranges $\RR_n$ and $\widetilde \RR_n$,
conditionally on the event $\{|\RR_n\cap \widetilde \RR_n|>\rho n\}$. 

\begin{theorem}\label{theo-appli-1}
Assume $d= 3$. There exist positive constants $c$ and $\kappa$, such that for any $n\ge 1$, $\rho\in (0,1)$, 
$\delta\in (0,1)$, and $r \le c\delta^{2/3}(\rho n)^{1/3}$,  
\begin{equation}\label{appli-2}
\lim_{n\to\infty}
\widetilde \Q_n^\rho \left[ | \RR_n\cap \widetilde \RR_n\cap 
\V_n( \delta \rho , r )|> (1-\kappa \delta)\, \rho n\right]=1.
\end{equation}
\end{theorem}

Recall that the heuristic picture is that as $\rho$
goes to zero, under the law $\widetilde \Q_n^\rho$, both random walks should localize
 in a region of typical diameter 
 $(n/\rho)^{1/3}$, during
their whole time-period. Thus, the occupation density in
the localization region is expected to be of order $\rho$, and
\reff{appli-2} provides a precise statement of this picture. 
Let us stress that unlike Theorem \ref{theo-appli-2}, the proof of this second result does not rely on BBH's fine Large Deviation Principle, but only on relatively soft arguments. 


Our main technical tool for proving both Theorems \ref{theo-appli-2} and \ref{theo-appli-1} is the following proposition, which 
allows us to estimate visits in a region of {\it low density} at a given
space scale $r$. 

\begin{proposition}\label{lemLambda}
Assume that $\Lambda$ is a subset of $\Z^d$ 
with the following property.
For some $\rho \in (0,1)$, and $r\ge 1$,
\begin{equation}\label{hypLambdaQ}
|\Lambda\cap Q(x,r)| \le \rho \cdot |Q(x,r)|, \quad \text{for all }x\in 2r\Z^d.
\end{equation}
There is a constant $\kappa>1$ independent of $r$, $\rho$, 
and $\Lambda$, such that for any $n\ge \rho^{\frac 2d-1}r^2$, 
and $t\ge \kappa \rho \, n$,
\begin{equation}\label{bruno-1}
\bP\left[|\RR_n\cap \Lambda|\ge t\right] \, 
\le\,  \exp\big(-\rho^{1-\frac 2 d}\, \frac{t}{2r^2}). 
\end{equation}
\end{proposition}
Note that in \reff{bruno-1} the smaller is the scale the
smaller is the probability. Note also that this result holds in any dimension $d\ge 3$.

Now some remarks on the limitation of our results are in order. 
Let us concentrate on 
Theorem~\ref{theo-appli-2} which is our main result. 
First the size of the time-window is constrained by the degree 
of precision in BBH's asymptotics. 
The fact that one can only consider windows of 
size order $\varepsilon n$, and not say order $\varepsilon^K n$, 
for some $K>1$, is related to the asymptotic  
of the rate function in the neighborhood 
of the origin, see \eqref{Imud} below. From 
\cite{BBH} one knows the first order term in dimension $3$. However, pushing further the precision of this  
asymptotic would allow to consider higher exponents $K$. On the other hand it would be even more  
interesting to allow time windows of smaller size, say of polynomial order $n^\kappa$, with $\kappa\in (2/3,1)$. Going below the exponent $2/3$ 
does not seem reasonable, as the natural belief is that strands of the path of length $n^{2/3}$ should typically move freely, 
and might visit from time to time regions with very low occupation 
density. Thus, we believe that a result in the same vein as
Theorem \ref{theo-appli-2} should hold for exponents 
$\kappa \in (2/3,1)$. One would need however a much better 
understanding of the speed of convergence in the 
Large Deviation Principle, see \eqref{BBH} below.

Similarly, one could ask whether our proof could show a kind of time inhomogeneity in dimension $5$ and higher. 
However, a problem for this 
is the following. Given two small time windows (say of order $\varepsilon n$), one would like to argue that the walk 
cannot visit regions with high occupation density in both, unless these two time windows were adjacent. However, even if they are not, the cost for 
the walk to come back at the origin at the beginning of each of them is only polynomially small, which is almost invisible when compared to the 
(stretched) exponentially small cost of the large deviations. Therefore obtaining such result seems out of reach at the moment.

Let us sketch the proof of Theorem~\ref{theo-appli-2}.
The first step, which reveals also the link between
\cite{BBH} and \cite{BBH2} is to show that 
$\{|\RR_n|-\E[|\RR_n|]< -\epsilon n\}$ implies large mutual intersections
$\{|\RR(I)\cap
\RR(I^c)|> \beta \epsilon |I|\}$ for some constant $\beta$ as
soon as the interval $I\subset [0,n]$ is large enough. 
This requires to show a LDP
on the same precision as \cite{BBH} for $\RR(I^c)$, where $I^c$
typically consists of two subintervals: this step presents some subtleties,
which we deal with by wrapping parts of the two trajectories, and use
that the intersection essentially increases under such operation.
Then one falls back on an estimate similar to Theorem~\ref{theo-appli-1}.

The rest of the paper is organized as follows. 
In Section~\ref{sec-tool} we gather useful results 
on the range: BBH's results, as well
as the upward large deviation principle of Hamana and Kesten \cite{HK},
and estimates on probability that the walk covers a 
fixed region with low density for a long time. We then prove 
Proposition~\ref{lemLambda} and Theorem \ref{theo-appli-1}.
In Section~\ref{sec-LD} we prove an extension of BBH's 
estimate when one considers two
independent walks starting from different positions. 
The proof of Theorem \ref{theo-appli-2} 
is concluded in Section~\ref{sec-theo2}.

\section{Visiting Sparse Regions}\label{sec-tool}
In this section, we prove our main tool, Proposition~\ref{lemLambda}, and then Theorem  \ref{theo-appli-1}, 
after we recall well known results.
\subsection{Preliminaries}
Dvoretzky and Erd\'os \cite{DE} established 
that there exists a constant $\kappa_d>0$, 
such that almost surely and in $L^1$, 
$$\lim_{n\to \infty}\ \frac{|\RR_n|}{n}= \kappa_d.$$
In addition, one has  
\begin{equation}\label{meanerror}
|\E[|\RR_n|] - \kappa_dn|= 
\left\{ \begin{array}{ll}
\OO(\sqrt n)& \text{when }d=3\\
\OO(\log n)&  \text{when }d=4\\
\OO(1)&  \text{when }d\ge 5.\\
\end{array}
\right.
\end{equation}
Precise asymptotic of the variance and a Central Limit Theorem 
are obtained by Jain and Orey \cite{JO} in dimensions $d\ge 5$, and
by Jain and Pruitt \cite{JP} in dimensions $d\ge 3$.

The analogue of the LDP of \cite{BBH} has been established
in the discrete setting by Phetpradap in his thesis \cite{P}, 
and reads as follows: 
there exists a function $I_d$, such that 
for any $\varepsilon \in (0,\kappa_d)$, 
\begin{equation}\label{BBH}
\lim_{n\to \infty}\, \frac{1}{n^{1-\frac 2d}}\log \bP\left[|\RR_n| - \E[|\RR_n|]\le -\varepsilon n \right] = -I_d(\varepsilon).
\end{equation}
Moreover, there exist positive constants $C$, $\mu_d$ and $\nu_d$, 
such that for $\varepsilon \in (0,\nu_d)$, 
\begin{equation}\label{Imud} 
\left.
\begin{array}{l}
\mu_3 \varepsilon^{2/3} \\
\mu_4\sqrt \varepsilon 
\end{array}
\right\}
\quad \le \, I_d(\varepsilon)\, \le \quad  \left\{
\begin{array}{ll}
\mu_3 \varepsilon^{2/3}(1+C\varepsilon) &  \text{when }d=3\\
\mu_4\sqrt \varepsilon (1+C \varepsilon^{1/3})  & \text{when }d= 4,
\end{array}
\right.
\end{equation}
and if $d\ge 5$, $I_d(\epsilon)=\mu_d \varepsilon^{1-\frac 2d}$.
These results were first obtained in \cite{BBH} for the Wiener sausage: 
the lower bounds are given by their Theorem 4 (ii) and Theorem 5 (iii), 
respectively for dimension $3,4$, and 
for dimension $5$ and higher (note that their constants $\mu_d$ differs from ours by a universal constant, see also \cite{P} for details). 
The upper bound in dimension $5$ and higher 
is also provided by their Theorem 5 (iii). 
The upper bound in dimension three and four is obtained 
in the course of the proof of Theorem 4 (ii), 
see their equations $(5.73)$, $(5.81)$ and $(5.82)$. 
Note that, as they use Donsker-Varadhan's large deviation theory,
their rate functions $I_d$ are given by variational formulas.

Now a basic fact we need is an estimate of the variance
of $|\RR_n\cap \Lambda|$ uniform on $\Lambda$.
\begin{lemma}\label{lem-var}
There exists a constant $C>0$, such that for any positive
integer $n$, and $\Lambda\subset \Z^d$, 
\begin{equation}\label{variance.range}
\text{Var}(|\RR_n\cap \Lambda|)\le C n\cdot \log^2 n.
\end{equation}
\end{lemma}
It is important to stress that $C$ is independent of $\Lambda$ and $n$.
This lemma is obtained using the same argument as in \cite[Lemma 6.2]{LG86}.

Upward large deviations are obtained
by Hamana and Kesten \cite{HK}. Their result implies 
that there exists a positive function $J_d$, such that for 
$\varepsilon\in (0,1-\kappa_d)$, 
\begin{equation}\label{HK}
\lim_{n\to \infty} \frac 1n \log \bP\left(|\RR_n| - \E[|\RR_n|] \ge \varepsilon n\right) = -J_d(\varepsilon). 
\end{equation}
Finally, an elementary fact we shall 
need is the following (see for instance Lemma 2.2 in \cite{AS17}).
\begin{lemma}\label{GreenV}
There exists a constant $C>0$, such that for any $\rho \in (0,1)$, $r\ge 1$, 
and $\Lambda\subset \Z^d$, satisfying \reff{hypLambdaQ}
one has for all $n\ge 1$, 
$$
\E[|\RR_n\cap \Lambda|]\le C(\rho^{2/d} r^2 + \rho n).
$$
\end{lemma}
\subsection{Proof of Proposition~\ref{lemLambda}}
Let $T:=\lfloor r^2/\rho^{1-\frac 2d}\rfloor$, 
and $R_j:=|\RR[jT,jT+T]\cap \Lambda|$. Note first that 
\begin{equation}\label{Prop.1}
|\RR_n\cap \Lambda|\le \sum_{j=0}^{\lfloor n/T\rfloor+1} R_j.
\end{equation}
Now, consider the martingale $(M_\ell)_{\ell \ge 0}$, defined by
$$
M_\ell:= \sum_{j=0}^\ell 
\Big( R_j- \E\left[R_j\mid \FF_{jT}\right]\Big).$$
By choosing $\kappa\ge 8C$, with $C$ as in Lemma~\ref{GreenV},  
we deduce from this Lemma, that for any $t\ge \kappa \rho\, n$, 
and $\rho n\ge \rho^{2/d}r^2$ that
$$
\sum_{j=0}^{\lfloor n/T\rfloor +1}  \E\left[R_j\mid \FF_{jT}\right] \le t/2.
$$
Hence, using \eqref{Prop.1} we get 
$$
\bP\left[|\RR_n\cap \Lambda|\ge t\right] \ 
\le\ \bP\left[M_{\lfloor k/T\rfloor +1}\ge t/2\right].
$$
Moreover, the increments of the martingale $(M_\ell)_{\ell \ge 0}$ are bounded by $T$, 
and by \eqref{variance.range} their conditional variance is 
$\mathcal O(T\log^2 T)$ (uniformly in $\Lambda$). 
Thus, McDiarmid's concentration 
inequality (see Theorem 6.1 in \cite{CL}), gives 
$$ 
\bP\left[M_{\lfloor n/T\rfloor+1}\ge t/2\right]\ \le\ \exp(-\frac{t}{2T}),
$$
by taking larger $\kappa$ if necessary. 
This proves the desired result. 
\qed

\subsection{Proof of Theorem~\ref{theo-appli-1}}\label{sec-theo1}
Fix $\rho\in (0,1)$ and $\delta \in (0,1)$. We proved in \cite{AS17} a lower bound for visiting a set
of density $\rho$ in dimension three: 
Proposition 4.1 of \cite{AS17} indeed establishes that
for some positive constants $c$ and $c'$,  
\begin{equation*}\label{lower-1}
\bP\big[ |\RR_n\cap \Lambda|> \rho |\Lambda|\big]\ge
\exp(-c'\rho^{2/3} n^{1/3}),\quad\text{for any }
\Lambda\subset B(0,c(n/\rho)^{1/3}).
\end{equation*}
By forcing one of the two walks to stay inside the desired ball,
we deduced that for some positive constant $c_0$ we have
the following rough bound on the intersection of two random walks:
\begin{equation}\label{lower-2}
\bP\big[ |\RR_n\cap \widetilde \RR_n| > \rho n\big]\ge
\exp(-c_0\, \rho^{2/3} n^{1/3}).
\end{equation}
Now, the (random) set $\RR_n\cap \V_n^c(\delta \rho, r)$
satisfies the hypothesis \reff{hypLambdaQ} of Proposition~\ref{lemLambda}
with density $\delta\rho$, thus giving with the constant $\kappa$ of this proposition 
\begin{equation*}\label{lower-3}
\bP\big[ |\RR_n\cap \widetilde\RR_n\cap \V_n^c(\delta \rho, r) |> \kappa\, \delta\rho\, n\big]
\le \exp\big(- (\delta \rho)^{1/3} \frac{\kappa\,  \delta\rho\, n}
{2 r^2}\big).
\end{equation*}
The proof follows after we observe that this probability becomes negligible, compared to the one in \eqref{lower-2}, if we choose $r$, satisfying
\begin{equation*}\label{lower-4}
r\le \big( \frac{\kappa}{2c_0}\big)^{1/2} 
\delta^{2/3}(\rho n)^{1/3}.
\end{equation*}

\section{Large deviation estimate for two random walks}\label{sec-LD}
In this section we prove an extension of \eqref{BBH}, when one considers two independent random walks starting from (possibly) different positions. 
We show that the upper bound in \eqref{BBH} still holds, up to a negligible factor, uniformly over all possible starting positions. 
While this result could presumably be also obtained by following carefully the proof of \cite{BBH}, 
we have preferred to follow here an alternative way and deduce it directly from \eqref{BBH}, using no heavy machinery.  
We state the result for dimension three only, since this is the case of interest for us here, but note that a similar result could be proved in any dimension $d\ge 3$, using exactly the same proof.

For $x\in \Z^3$, we denote by $\bP_{0,x}$ the law of two independent random walks $S$ and $\widetilde S$ starting respectively from the origin and $x$. We write $\RR$ and $\widetilde \RR$ for their ranges. Furthermore, for an integrable random variable $X$, we set $\overline X:=X-\E[X]$, and for $x\in \Z^3$, we denote by $\|x\|$ its Euclidean norm. 

\begin{proposition}\label{propBBH} Assume that $d= 3$. 
There exists $\varepsilon_0>0$, such that for any $\varepsilon \in (0,\varepsilon_0)$, for $n$ large enough, and $k\le n$, 
$$\sup_{\|x\|\le n^{2/3}}\, \bP_{0,x}\left[\overline{|\RR_k\cup \widetilde \RR_{n-k}|}\le - \varepsilon n \right] \le \exp\left(-I_3( \frac{\varepsilon}{1+\varepsilon^2})n^{1/3}\right).$$ 
\end{proposition}

\begin{proof}
Set $m:=\lfloor \varepsilon^2 n\rfloor$. First using \eqref{BBH}, we know that 
\begin{align*}
\lim_{n\to\infty}\, \frac{1}{n^{1/3}} \log \bP\left(\overline{|\RR_{n+m}|}\le -\varepsilon(1-\varepsilon^3) n\right)=-I_3(\frac{\varepsilon(1-\varepsilon^3)}{1+\varepsilon^2})(1+\varepsilon^2)^{1/3}.
\end{align*}
Now consider $x$, with $\|x\|\le \varepsilon^{5/2} n^{2/3}$, 
and $k\le n$. Note that by \eqref{meanerror}, we get for $\varepsilon>0$,  
\begin{equation*}
\overline{|\RR[0,n+m]|}\le 
\overline{|\RR[0,k]\cup \RR[k+m,n+m]|} + 
\overline{|\RR[k,k+m]} + \mathcal O(\sqrt n).
\end{equation*}
Therefore, at least for $n$ large enough, 
\begin{align*}
&\bP\left(\overline{|\RR_{n+m}|}\le -\varepsilon(1-\varepsilon^3) n\right)\\
&\ge\ \bP\left( \overline{|\RR_k\cup \RR[k+m,n+m]|}\le -\varepsilon n, \overline{|\RR[k,k+m]|} 
\le \varepsilon^5 n, S_{k+m} - S_k=x \right)\\
&\ge\ \bP_{0,x}\left( \overline{|\RR_k\cup \widetilde \RR_{n-k} |}\le -\varepsilon n\right) \cdot \bP\left(\overline{|\RR_m|} \le \varepsilon^5 n, \, S_m =x \right),
\end{align*}
using the Markov property and reversibility of the random walk for the last inequality.

Now using Hamana and Kesten bound \eqref{HK}, 
the local central limit theorem (see Theorem 2.3.11 in \cite{LL}), 
and that $\|x\|\le \varepsilon^{5/2} n^{2/3}$, we get that for some constant $c>0$ (independent of $x$), and for all $n$ large enough, 
\begin{align*}
 \bP\left(\overline{|\RR_m|} \le \varepsilon^5 n, S_{m} =x \right)&\ge \bP(S_{m} = x) - \bP(\overline{|\RR_{m}|} \ge \varepsilon^5 n) \\
 &\ge \exp(-c \varepsilon^3 n^{1/3}) - \exp(-J_d(\varepsilon^3)(1-\varepsilon) \varepsilon^2n )\\
 &\ge \frac 12 \exp(-c\varepsilon^3 n^{1/3}).
\end{align*}
Moreover, it follows from \eqref{Imud}, that for $\varepsilon$ small enough 
$$I_3( \frac{\varepsilon(1-\varepsilon^3)}{1+\varepsilon^2})(1+\varepsilon^2)^{1/3} - c\varepsilon^3 \ge I_3( \frac{\varepsilon}{1+\varepsilon^2})(1+\frac 14\varepsilon^2).$$
Therefore, for $\varepsilon$ small enough, and then for all $n$ large enough, 
\begin{equation}\label{petitx}
\sup_{\|x\|\le \varepsilon^{5/2}n^{2/3}}  \bP_{0,x}\left( \overline{|\RR_k\cup \widetilde \RR_{n-k} |}\le -\varepsilon n\right) \le 
\exp\left(-I_3( \frac{\varepsilon}{1+\varepsilon^2})(1+\frac 15\varepsilon^2) n^{1/3}\right).
\end{equation}
It remains to consider $x$ satisfying $\varepsilon^{5/2} n^{2/3}\le \|x\|\le n^{2/3}$. 
Our strategy is to show that there exists a constant $\rho\in (0,1)$, such that for any such $x$, 
the probability of the event 
\begin{equation}\label{A}
\mathcal A:=\{\overline{|\RR_k\cup \widetilde \RR_{n-k}|}\le - \varepsilon n\},
\end{equation}
under $\bP_{0,x}$ is bounded by the probability of the same event under $\bP_{0,x'}$, with $x'$ satisfying $\|x'\|\le \rho \|x\|$, up to some negligible error term. Applying this estimate at most order $\log (1/\varepsilon)$ times, and using \eqref{petitx}, the result obtains.

Fix $x$ such that $\varepsilon^{5/2} n^{2/3}\le \|x\|\le n^{2/3}$, 
and assume, without loss of generality, that its largest coordinate 
in absolute value is the first one, say $x_1$, and that it is positive. 
In this case $\|x\|\ge x_1 \ge \|x\|/\sqrt 3$. 
Recall next that we consider two random walks $S$ and $\widetilde S$, 
one starting from the origin and running up to time $k$,
 and the other one starting from $x$ and running up to time $n-k$.  
For $x_1/3\le y\le 2x_1/3$, consider the hyperplane $\mathcal H_y$ of vertices having first coordinate equal to $y$. 
Denote by $N_y$ the number of excursions out of $\mathcal H_y$, made by one of the two walks $S$ or $\widetilde S$, 
which hit the hyperplane $\mathcal H_{y+[\varepsilon^{-10}]}$. Note that one can order them by time of arrival, considering first those made by $S$ 
and then those made by $\widetilde S$. 
Let $N'_y$ be the number of excursions 
among the first $N_y\wedge (2\varepsilon^3 n^{1/3})$ previous ones, which spend a time at least $\varepsilon^{-17}$ in the region 
between $\mathcal H_y$ and $\mathcal H_{y+[\varepsilon^{-10}]}$. Note that for a single excursion,  
the probability to hit $\mathcal H_{y+[\varepsilon^{-10}]}$ in less than $\varepsilon^{-17}$ steps, is of order $\exp(-c\varepsilon^{-3})$, for some constant $c>0$. By independence between the first $2\varepsilon^3 n^{1/3}$ excursions, we deduce that 
\begin{equation}\label{Ny'}
\bP_{0,x}(N'_y\le \varepsilon^3 n^{1/3}, N_y\ge 2\varepsilon^3 n^{1/3})\le \exp(- c n^{1/3}),
\end{equation}
for some possibly smaller constant $c>0$. On the other hand, let $T_y$ be the cumulated total time spent by $S$ and $\widetilde S$ 
in the region between $\mathcal H_y$ and $\mathcal H_{y+[\varepsilon^{-10}]}$. 
Observe that the number of levels $y$ between $x_1/3$ and $2x_1/3$ which are integer multiples of $[\varepsilon^{-10}]$ is of order $\varepsilon^{10} x_1/3$, and that the latter is (at least for $\varepsilon$ small enough) larger than $\varepsilon^{13}n^{2/3}$. Thus, for at least one such $y$, one must have both $N'_y \le \varepsilon^3 n^{1/3}$ and $T_y\le \varepsilon^{-13} n^{1/3}$ (otherwise the total time spent in the region between the hyperplanes $\mathcal H_{x_1/3}$ and $\mathcal H_{2x_1/3}$ would exceed $n$).     
Using \eqref{Ny'}, we deduce that 
$$\bP_{0,x}\left(N_y\ge 2\varepsilon^3 n^{1/3}\text{ or }T_y\ge\varepsilon^{-13}n^{1/3},  \quad \text{for all } y\in \{x_1/3,\dots,2x_1/3\}\right)\le x_1\exp(-cn^{1/3}),$$
for some constant $c>0$. Then as a consequence of the pigeonhole principle, there exists (a deterministic) $y_0\in  \{x_1/3,\dots,2x_1/3\}$, such that 
$$\bP_{0,x}(N_{y_0}\le 2\varepsilon^3 n^{1/3}, T_{y_0}\le \varepsilon^{-13}n^{1/3}, \mathcal A)\ge \frac 1{x_1}\bP_{0,x}(\mathcal A)-\exp(-cn^{1/3}),$$
with the event $\mathcal A$ as defined in \eqref{A}.

Denote now by $x'$ the symmetric of $x$ with respect to $\mathcal H_{y_0}$. First observe that since $x_1\ge \|x\|/\sqrt 3$, and $x_1/3\le y_0\le 2x_1/3$, there exists $\rho\in (0,1)$ (independent of $x$), such that $\|x'\|\le \rho \|x\|$. 
Next recall that any excursion out of $\mathcal H_{y_0}$ and its symmetric with respect to $\mathcal H_{y_0}$ have the same probability to happen, and similarly for the first part of the trajectory of $\widetilde S$, up to the hitting time, say $\tau$, of $\mathcal H_{y_0}$, under the law $\bP_{0,x'}$. 
Moreover, by reflecting (with respect to $\mathcal H_{y_0}$) all the excursions out of $\mathcal H_{y_0}$ which hit 
$\mathcal H_{y_0+[\varepsilon^{-10}]}$, plus $\widetilde S[0,\tau]$, 
one can increase the size of the range by at most $T_{y_0}$. Therefore,  
$$\bP_{0,x'}(N_{y_0}=0, \overline{|\RR_k\cup \widetilde \RR_{n-k}|}\le - \varepsilon n+\varepsilon^{-13}n^{1/3} ) \ge 2^{-2\varepsilon^3 n^{1/3}} \bP_{0,x}(N_{y_0}\le 2\varepsilon^3 n^{1/3}, T_{y_0}\le \varepsilon^{-13}n^{1/3}, \mathcal A).$$
Combining the last two displays we conclude that for $n$ large enough, 
$$\bP_{0,x}(\mathcal A) \le  2^{4\varepsilon^3 n^{1/3}}\bP_{0,x'}\left(\overline{|\RR_k\cup \widetilde \RR_{n-k}|}\le - \varepsilon n+\varepsilon^{-13}n^{1/3} \right) + \exp(-cn^{1/3}),$$
for some (possibly smaller) constant $c>0$. 
Repeating the same argument $(5/2)\log \varepsilon/(\log \rho)$ times, and using \eqref{petitx} and \eqref{Imud}, we obtain the desired result. 
\end{proof}

\section{Proof of Theorem \ref{theo-appli-2}}\label{sec-theo2}
Let $k\le\ell \le n$ be given, satisfying $\ell - k = \lfloor K_0\varepsilon\, n\rfloor$, with  
$K_0$ a constant to be fixed later. Write $I=\{k,\dots,\ell\}$. Using that  
$$
\overline{|\RR_n|}=\overline{|\RR(I)|}+
 \overline{|\RR(I^c)|} - \overline{|\RR(I)\cap \RR(I^c)|},
$$
we have, with $\nu:=1-\frac{K_0\varepsilon}{3}$, 
\begin{equation}
\label{R1R2}
\mathbb P(\overline{|\RR_n|} \le -\varepsilon n)\ \le\ \mathbb P(\overline{|\RR(I)|} +  \overline{|\RR(I^c)|}\le -\nu \varepsilon n) 
+ \mathbb P(|\RR(I)\cap \RR(I^c)| \ge \frac {K_0}3 \varepsilon^2 n).
\end{equation}
We start by showing that the first probability on the right-hand side is negligible 
(when compared to the probability on the left-hand side). For this let 
$N=\lfloor \varepsilon^{-2}\rfloor$, and for $i=0,\dots, N$, let $\alpha_i := i\varepsilon^2$. Then, note that for $n$ large enough, 
\begin{align}\label{R1R2.2}
\nonumber & \mathbb P(\overline{|\RR(I)|} +  \overline{|\RR(I^c)|} 
\le -\nu \varepsilon n)\ \le\  \sum_{i=0}^N \mathbb P\left(\overline{|\RR(I)|}\le -\nu \alpha_i \varepsilon n,  \overline{|\RR(I^c)|}\le -\nu(1-\alpha_{i+1})\varepsilon n\right)\\
\nonumber&+\mathbb P\left(\overline{|\RR(I^c)|}\le -\nu \varepsilon n\right)
\le\  \sum_{i=0}^N \sum_{\|x\|\le n^{2/3}} \mathbb P\left(\overline{|\RR(I)|}\le -\nu \alpha_i \varepsilon n, \, \overline{|\RR(I^c)|}\le -\nu (1-\alpha_{i+1})\varepsilon n,\, S_{\ell } - S_k= x\right)\\
&\qquad +\mathbb P\left(\overline{|\RR(I^c)|}\le -\nu \varepsilon n\right) + \mathcal O(\exp(-cn^{1/3})), 
\end{align}
using Chernoff's bound for the last inequality (see for instance Theorem 3.1 in \cite{CL}). 
Now applying the Markov property and using Proposition \ref{propBBH} and \eqref{BBH}, we get that for any $i\le N$, there exists $n_i\ge 1$, such that for all $n\ge n_i$, 
\begin{align}\label{R1R2.3}
\nonumber &\sum_{\|x\|\le n^{2/3}}  \mathbb P\left(\overline{|\RR(I)|}\le -\nu \alpha_i \varepsilon n, \, \overline{|\RR(I^c)|}\le -\nu(1-\alpha_{i+1})\varepsilon n,\, 
S_{\ell} - S_k = x\right)\\
\nonumber & \le \ \sum_{\|x\|\le n^{2/3}}  \mathbb P\left(\overline{|\RR(I)|}\le -\nu \alpha_i \varepsilon n,\, S_{\ell} -S_k= x\right) \cdot \exp\left(-I_3(\frac{\varepsilon_i}{1+\varepsilon_i^2}) (1-K_0\varepsilon)^{1/3}n^{1/3}\right)\\
&\le \exp\left(-I_3(\widetilde \varepsilon_i) (1-\varepsilon^2) (K_0\varepsilon n)^{1/3}-I_3(\frac{\varepsilon_i}{1+\varepsilon_i^2}) (1-K_0\varepsilon)^{1/3}n^{1/3}\right),
\end{align}
with 
$$\varepsilon_i := \nu \frac{(1-\alpha_{i+1})\varepsilon}{1-K_0\varepsilon},\qquad \text{and}\qquad \widetilde \varepsilon_i:=\frac{\nu \alpha_i}{K_0}.
$$
Note that by choosing larger $n_i$ if necessary, one can also assume that \eqref{R1R2.2} holds for $n\ge n_i$.  
Note furthermore that by \eqref{Imud}, the first term in the exponential in \eqref{R1R2.3} is already larger than $2I_3(\varepsilon)n^{1/3}$, when $\alpha_i\ge K_0\sqrt \varepsilon$, at least provided $K_0$ is large enough and $\varepsilon$ small enough. 
Thus in the following one can assume that $\alpha_i\le K_0\sqrt \varepsilon$. We will also assume that $\varepsilon$ is small enough, so that $K_0\sqrt \varepsilon<1/2$. 
Then, using in particular that $\alpha_{i+1} = \alpha_i + \varepsilon^2$, we get 
(recall that $\nu=1-\frac{K_0\varepsilon}{3}$)
\begin{align*}
&I_3(\widetilde \varepsilon_i) (1-\varepsilon^2) (K_0\varepsilon)^{1/3}+I_3(\frac{\varepsilon_i}{1+\varepsilon_i^2}) (1-K_0\varepsilon)^{1/3}\\
&\ge \ \mu_3(1-\frac{K_0\varepsilon}{3})^{2/3}\left( \frac {\alpha_i^{2/3}}{K_0^{1/3}\varepsilon^{1/3}} + \frac{(1-\alpha_i)^{2/3}}{(1-K_0\varepsilon)^{1/3}}\right)\varepsilon^{2/3} - \mathcal O(\varepsilon^2).
\end{align*}
Now we claim that the bound in the parenthesis above reaches its infimum when $i=0$, or equivalently when $\alpha_i=0$. 
To see this, it suffices to consider the variations of the function $f$ defined for $u\in (0,1)$ by $f(u)=c_1^{1/3}u^{2/3} + c_2^{1/3}(1-u)^{2/3}$, with $c_1=(K_0\varepsilon)^{-1}$,  
and $c_2= (1-K_0\varepsilon)^{-1}$. 
A straightforward computation shows that $f'(u)>0$ on $(0,u_0)$,   
with $u_0= 1-K_0\varepsilon$. Since we assumed that $u_0>1/2>K_0\sqrt \varepsilon$, this proves our claim. By taking $K_0=100 C$, 
with $C$ the constant appearing in the upper bound of $I_3$ in \eqref{Imud}, one deduces that for $\varepsilon$ small enough, 
\begin{align*}
I_3(\widetilde \varepsilon_i) (1-\varepsilon^2) (K_0\varepsilon)^{1/3}+I_3(\frac{\varepsilon_i}{1+\varepsilon_i^2}) (1-K_0\varepsilon)^{1/3}
 \ge \ \mu_3(1+\frac{K_0\varepsilon}{10}) \varepsilon^{2/3} - \mathcal O(\varepsilon^2) \ge I_3(\varepsilon) (1+\frac{K_0\varepsilon}{20}). 
\end{align*}
As a consequence, letting $N(\varepsilon):=\max_i n_i$, we get that for $\varepsilon$ small enough, for all $n\ge N(\varepsilon)$, 
$$\sum_{i=0}^N \mathbb P\left(\overline{|\RR(I)|}\le -\nu \alpha_i \varepsilon n,  \overline{|\RR(I^c)|}\le -\nu(1-\alpha_{i+1})\varepsilon n\right) 
\le \exp(- I_3(\varepsilon)(1+\frac{K_0\varepsilon}{20})n^{1/3}).$$
One can obtain similarly the same bound for the term $\mathbb P(\overline{|\RR(I^c)|}\le -\nu \varepsilon n)$, also for all $n\ge N(\varepsilon)$, possibly by taking a larger constant $N(\varepsilon)$ if necessary. Finally using again \eqref{BBH}, we get that for all $\varepsilon$ small enough,
\begin{align}\label{R1R2.4}
\mathbb P(\overline{|\RR(I)|} +  \overline{|\RR(I^c)|} 
\le -\nu \varepsilon n) = o\left(\frac 1n\mathbb P(\overline{|\RR_n|}\le -\varepsilon n)\right).
\end{align}
It remains to estimate the second term in the right-hand side of \eqref{R1R2}. This is similar to the proof of Theorem \ref{theo-appli-1}. Assume given $1\le r\le \varepsilon^{5/6} n^{1/3}$, and $\beta>0$, whose value will be made more precise in a moment. 
To simplify notation, 
write $\RR_1=\RR[0,k]$ and $\RR_2:=\RR[\ell,n]$. Next set 
$$
\Lambda_1 = \cup_{x\in \mathcal C_1} Q(x,r),\qquad \text{and}\qquad \Lambda_2= \cup_{x\in \mathcal C_2} Q(x,r),$$
with 
$$\mathcal C_1:=\{x\in 2r\Z^d\ : \ |Q(x,r)\cap \RR_1|\ge \beta \varepsilon r^d\}, \quad \text{and}\quad \mathcal C_2:=\{x\in 2r\Z^d\ : \ |Q(x,r)\cap \RR_2|\ge \beta \varepsilon r^d\}.$$
Since $\RR(I^c)=\RR_1 \cup \RR_2$, one has 
$$
|\RR(I)\cap \RR(I^c)|\le |\RR(I)\cap \RR_1|+|\RR(I)\cap \RR_2|,
$$ 
and therefore, 
$$\mathbb P(|\RR(I)\cap \RR(I^c)| \ge \frac{K_0}{3}\varepsilon^2n)\le \mathbb P(|\RR(I)\cap \RR_1| \ge \frac{K_0}{6}\varepsilon^2n)+\mathbb P(|\RR(I)\cap \RR_2| \ge \frac{K_0}{6}\varepsilon^2n).$$
Both terms on the right-hand side are treated similarly. 
We first fix $\beta< 1/(24\kappa)$, with $\kappa$ the constant appearing in statement of Lemma \ref{lemLambda}. 
Then applying Lemma \ref{lemLambda} with $\rho=\beta \varepsilon$, 
$n=\lfloor K_0\varepsilon n\rfloor $, and $t=\frac{K_0}{24}\varepsilon^2 n$, we get (using also the Markov property at time $k$), 
\begin{align*}
\mathbb P(|\RR(I)\cap \RR_1\cap \Lambda_1^c | \ge \frac{K_0}{24}  \varepsilon^2 n)\ &\le\  \exp(-(\beta \varepsilon)^{1/3} \frac{K_0\varepsilon^2n}{48 r^2})\\
&\le \ \exp(- \frac{\beta^{1/3}K_0}{48} \varepsilon^{2/3} n^{1/3}),
\end{align*}
using that $r\le \varepsilon^{5/6} n^{1/3}$, for the last inequality. By taking larger $K_0$ if necessary, one can ensure that this bound is 
$o((1/n)\cdot\mathbb P( \overline{|\RR_n|}\le -\varepsilon n))$. This way we obtain 
\begin{equation}\label{R1R2.5}
\mathbb P(|\RR(I)\cap \RR_1\cap \Lambda_1^c | \vee |\RR(I)\cap \RR_2\cap \Lambda_2^c |\ge \frac{K_0}{24}  \varepsilon^2 n) =o\left(\frac 1n \mathbb P( \overline{|\RR_n|}\le -\varepsilon n)\right).
\end{equation}
Coming back to \eqref{R1R2}, dividing both sides of the inequality by the term on the left-hand side, and using \eqref{R1R2.4} and \eqref{R1R2.5}, 
we get that for all $\varepsilon$ small enough, 
$$\mathbb Q_n^\varepsilon\left(|\RR(I)\cap \Lambda_1\cap \RR_1|\ge \frac{K_0}{8}\varepsilon^2 n,\quad \text{ or }\quad |\RR(I)\cap \Lambda_2\cap \RR_2|\ge \frac{K_0}{8}\varepsilon^2 n\right)\ge 1-o\left(\frac 1n\right).$$
Since both
$$
\Lambda_1\cap \RR_1\ \subseteq \ \V_n(\beta \varepsilon,r)
\cap \RR(I^c),\quad \text{and}\quad 
\Lambda_2\cap \RR_2\ \subseteq\ \V_n(\beta \varepsilon,r)
\cap \RR(I^c),
$$
we get 
$$\mathbb Q_n^\varepsilon\left(|\RR(I)\cap \RR(I^c) \cap \V_n(\beta \varepsilon,r)|\ge \frac{K_0}{8}\varepsilon^2 n\right)\ge 1-o\left(\frac 1n\right).$$
The proof of Theorem \ref{theo-appli-2} follows by a union bound, since there are at most $n$ intervals $I$ of fixed length in $I_n$.  \hfill $\square$


\begin{thebibliography}{99}

\bibitem[AS17]{AS17} A. Asselah and B. Schapira, {\it Moderate deviations for the range of a random walk: path concentration},  Ann. Sci. \'Ec. Norm. Sup\'er. (4) 50 (2017), no. 3, 755--786. 

\bibitem[BBH01]{BBH} M. van den Berg, E. Bolthausen, F. den Hollander, {\it Moderate deviations for the volume of the Wiener sausage}. Ann. of Math. (2) 153 (2001), no. 2, 355--406.

\bibitem[BBH04]{BBH2} M. van den Berg, E. Bolthausen, F. den Hollander, {\it On the volume of the intersection of two Wiener sausages}. 
Annals of Mathematics 159 (2004), no. 2, 741--782.

\bibitem[CL06]{CL}  F. Chung, L. Lu, {\it Concentration inequalities and martingale inequalities: a survey}. Internet Math. 3 (2006), no. 1, 79--127.

\bibitem[DE51]{DE}  A. Dvoretzky, P. Erd\`os, {\it Some problems on random walk in space}. Proceedings of the Second Berkeley Symposium on Mathematical Statistics and Probability, 1950. pp. 353--367. University of California Press, Berkeley and Los Angeles, 1951. 

\bibitem[HK01]{HK} Y. Hamana, H. Kesten, {\it A large-deviation result for the range of random walk and for the Wiener sausage}. 
Probab. Theory Related Fields 120 (2001), no. 2, 183--208. 

\bibitem[JP71]{JP} N. C. Jain, W. E. Pruitt, {\it The range of transient random walk}, J. Analyse Math. 24 (1971), 369--393.

\bibitem[JO69]{JO}  N. C. Jain, S. Orey, {\it On the range of random walk}, Israel J. Math. 6 1968 373--380 (1969).

\bibitem[LG86]{LG86}  J.-F. Le Gall, {\it Propri\'et\'es d'intersection des marches al\'eatoires. I. Convergence vers le temps local d'intersection}, Comm. Math. Phys. 104 (1986), no. 3, 471--507.  

\bibitem[LL10]{LL}  G. F. Lawler, V. Limic, {\it Random walk: a modern introduction}, Cambridge Studies in Advanced Mathematics, 123. Cambridge University Press, Cambridge, 2010. xii+364 pp.

\bibitem[Phet11]{P} P. Phetpradap. {\it Intersections of Random Walks}, Ph.D. thesis, Univ. Bath (2011), available on http://opus.bath.ac.uk/view/person\_id/3186.html. 

\bibitem[S17]{S} A.-S. Sznitman, {\it Disconnection, random walks, and random interlacements}, 
Probab. Theory Related Fields 167 (2017), no. 1-2, 1--44.
\end{thebibliography}
\end{document}